\documentclass[10pt,twoside]{article}

\usepackage{amsmath}
\usepackage{amsfonts}
\usepackage{amssymb}
\usepackage{amscd}
\usepackage{amsthm}
\usepackage{amsbsy}
\usepackage{graphicx}
\usepackage{bm}
\usepackage[all]{xy}
\def\@oddhead{\hfill \shorttitle \hfill \thepage}
\def\@evenhead{\thepage \hfill \shortauthor \hfill}
\def\@oddfoot{}
\def\@evenfoot{}

\newtheorem{thm}{Theorem}[section]

\newtheorem{lem}[thm]{Lemma}
\newtheorem*{lem*}{Lemma}

\newtheorem{coro}[thm]{Corollary}

\theoremstyle{definition}

\newtheorem{de}[thm]{Definition}

\theoremstyle{remark}

\def\o{\circ}
\def\X{\mathfrak X}
\def\al{\alpha}
\def\be{\beta}
\def\ga{\gamma}
\def\de{\delta}
\def\ep{\varepsilon}
\def\ze{\zeta}
\def\et{\eta}

\def\ka{\kappa}
\def\la{\lambda}
\def\rh{\rho}
\def\si{\sigma}

\def\ph{\varphi}

\def\om{\omega}
\def\Ga{\Gamma}
\def\De{\Delta}

\def\Ph{\Phi}

\def\Om{\Omega}
\def\i{^{-1}}
\def\x{\times}
\def\p{\partial}
\let\on=\operatorname
\def\L{\mathcal L}
\def\Diff{\on{Diff}}
\def\Emb{\on{Emb}}
\def\Imm{\on{Imm}}
\def\vol{\on{vol}}
\newcommand{\East}[2]{-\raisebox{0.1pt}{$\mkern-16mu\frac{\;\;#1\;}{\;\;#2\;}\mkern-16mu$}\to}

\allowdisplaybreaks

\title{\ \\[0.4cm] \ \\ \bf  Manifolds of mappings and shapes }
\author{Peter W.\ Michor\footnote{Fakult\"at f\"ur Mathematik,
Universit\"at Wien, Os\-kar-Mor\-gen\-stern-Platz 1, A-1090 Wien, Austria.
E-mail: peter.michor@univie.ac.at.}}

\begin{document}

\maketitle


\thispagestyle{empty}

\begin{abstract}
\vskip 3mm\footnotesize{

\vskip 4.5mm
\noindent
In his Habilitationsvortrag, Riemann described infinite dimensional manifolds parameterizing 
functions and shapes of solids. This is taken as an excuse to describe convenient calculus in 
infinite dimensions which allows for short and transparent proofs of the main facts of the theory 
of manifolds of smooth mappings. Smooth manifolds of immersions, diffeomorphisms, and shapes, and weak Riemannian metrics on them are treated, culminating in 
the surprising fact, that geodesic distance can vanish completely for them. 

\vspace*{2mm} \noindent{\bf 2010 Mathematics Subject Classification:
}  58B20, 58D15, 35Q31

\vspace*{2mm} \noindent{\bf Keywords and Phrases: }}
Convenient calculus, Manifolds of mappings, Diffeomorphism groups, Shape spaces, weak Riemannian 
metrics. 

\end{abstract}
\vspace*{-11.2cm}
%
\vspace*{11.6cm}

\section*{Contents}

1 Introduction
\\ \\
2 A short review of convenient calculus in infinite dimensions
\\ \\
3 Manifolds of mappings and regular Lie groups
\\ \\
4 Weak Riemannian metrics
\\ \\
5 Robust weak Riemannian metrics and Riemannian submersions

\newpage

\section{Introduction}\label{intro}
At the very birthplace of the notion of manifolds, in the 
Habilitations\-schrift \cite[end of section I]{Riemann1854}, Riemann mentioned infinite dimensional 
manifolds explicitly: 
\begin{quote}
``Es giebt indess auch Mannigfaltigkeiten, in welchen die Ortsbestimmung
nicht eine endliche Zahl, sondern entweder eine unendliche Reihe
oder eine stetige Mannigfaltigkeit von Gr\"os\-sen\-be\-stimmun\-gen
erfordert.
Solche Mannigfaltigkeiten bilden z.B.\ die m\"oglichen Bestimmungen einer
Function f\"ur ein gegebenes Gebiet, die m\"oglichen Gestalten einer
r\"aumlichen Figur u.s.w.''
\end{quote}
The translation into English in \cite{RiemannClifford} reads as follows:
\begin{quote}
``There are manifoldnesses in which the 
determination of position requires not a finite number, but either an endless series or a continuous 
manifoldness of determinations of quantity. Such manifoldnesses are, for example, the possible 
determinations of a function for a given region, the possible shapes of a solid figure, \&c.''
\end{quote} 
 
If one reads this with a lot of good will one can interpret it as follows: When Riemann sketched 
the general notion of a manifold, he also forsaw the notion of an infinite dimensional manifold of 
mappings between manifolds, and of a manifold of shapes. He then went on to describe the notion of 
Riemannian metric and to talk about curvature. 
I will take this as an excuse to describe 
the theory of manifolds of mappings, of diffeomorphisms, and of shapes, and of some striking 
results about weak Riemannian geometry on these spaces. See \cite{Bauer2014} for an overview 
article which is much more comprehensive for the aspect of shape spaces.

An explicit construction of manifolds of smooth mappings modeled on Fr\'echet spaces 
was described by  \cite{Eells58}. Differential calculus beyond the realm of Banach spaces has some 
inherent difficulties even in its definition; see section \ref{calc}.
Smoothness of composition and inversion was first treated on the group of all smooth 
diffeomorphisms of a compact manifold in \cite{Leslie67}; however, there was a gap in the proof, 
which was first filled by \cite{Gutknecht77}. 
Manifolds of $C^k$-mappings and/or mappings of Sobolev classes were treated by 
\cite{Eliasson67}, \cite{Eells66}, Smale-Abraham \cite{Abraham62}, and \cite{Palais68}. Since these 
are modeled on Banach spaces, they allow solution methods for equations and have found a lot of 
applications. See in particular \cite{EbinMarsden70}. 

\section{A short review of convenient calculus in infinite dimensions}\label{calc}

Traditional differential calculus works 
well for finite dimensional vector spaces and for Banach spaces. 
Beyond Banach spaces,
the main difficulty is that composition of 
linear mappings stops to be jointly continuous at the level of Banach 
spaces, for any compatible topology. Namely, if for a locally convex vector space $E$ and its dual 
$E'$ the evaluation mapping $\on{ev}:E\x E'\to \mathbb R$ is jointly continuous, then there are open 
neighborhoods of zero $U\subset E$ and $U'\subset E'$ with $\on{ev}(U\x U')\subset [-1,1]$. But then 
$U'$ is contained in the polar of an open set, and thus is bounded. So $E'$ is normable, and a 
fortiori $E$ is normable. 

For locally convex spaces which are more general than Banach spaces,
we sketch here the convenient approach as explained in 
\cite{FrolicherKriegl88} and \cite{KrieglMichor97}.

The name {\it convenient calculus} mimicks the paper \cite{Steenrod67}
whose results (but not the name `convenient') was predated by
\cite{Brown61}, \cite{Brown63}, \cite{Brown64}.
They discussed compactly generated spaces as a cartesian closed category for algebraic topology.  
Historical remarks on only those developments of calculus beyond Banach spaces that led to convenient 
calculus are given in \cite[end of chapter I, p.\ 73ff]{KrieglMichor97}.

\subsection{The $c^\infty$-topology}
Let $E$ be a 
locally convex vector space. A curve $c:\mathbb R\to E$ is called 
{\it smooth} or $C^\infty$ if all derivatives exist and are 
continuous. Let 
$C^\infty(\mathbb R,E)$ be the space of smooth curves. It can be 
shown that the set $C^\infty(\mathbb R,E)$ does not entirely depend on the locally convex 
topology of $E$, only on its associated bornology (system of bounded sets); see 
\cite[2.11]{KrieglMichor97}.
The final topologies with respect to the following sets of mappings into E coincide; see 
\cite[2.13]{KrieglMichor97}:
\begin{enumerate}
\item $C^\infty(\mathbb R,E)$.
\item The set of all Lipschitz curves 
(so that $\{\frac{c(t)-c(s)}{t-s}:t\neq s, |t|,|s|\le C\}$ 
is bounded in $E$, for each $C$). 
\item The set of injections $E_B\to E$ where $B$ runs through all bounded 
absolutely convex subsets in $E$, and where 
$E_B$ is the linear span of $B$ equipped with the Minkowski 
functional $\|x\|_B:= \inf\{\la>0:x\in\la B\}$.
\item The set of all Mackey-convergent sequences $x_n\to x$ 
(there exists a sequence 
$0<\la_n\nearrow\infty$ with $\la_n(x_n-x)$ bounded).
\end{enumerate}
{\it This topology is called the $c^\infty$-topology on $E$ and we write 
$c^\infty E$ for the resulting topological space.} 

In general 
(on the space $\mathcal{D}$ of test functions for example) it is finer 
than the given locally convex topology, it is not a vector space 
topology, since addition is no longer jointly 
continuous. Namely, even 
$c^\infty (\mathcal D\x \mathcal D)\ne c^\infty\mathcal D\x c^\infty\mathcal D$.

The finest among all locally convex topologies on $E$ 
which are coarser than $c^\infty E$ is the bornologification of the 
given locally convex topology. If $E$ is a Fr\'echet space, then 
$c^\infty E = E$. 

\subsection{Convenient vector spaces} 
A locally convex vector space 
$E$ is said to be a {\it convenient 
vector space} if one of the following equivalent conditions holds
(called $c^\infty$-completeness); see \cite[2.14]{KrieglMichor97}:
\begin{enumerate}
\item For any $c\in C^\infty(\mathbb R,E)$ the (Riemann-) integral 
$\int_0^1c(t)dt$ exists in $E$.
\item Any Lipschitz curve in $E$ is locally Riemann integrable.
\item A curve $c:\mathbb R\to E$ is $C^\infty$ if and only if $\la\o c$ is $C^\infty$
for all $\la\in E^*$, where $E^*$ is the dual  
of all continuous linear functionals on $E$. 
\begin{itemize}
\item
Equivalently, for all 
$\la\in E'$, the dual  of all bounded linear functionals. 
\item
Equivalently, for all $\la\in \mathcal V$, where $\mathcal V$ is a subset of $E'$ which 
recognizes bounded subsets in $E$. 
\end{itemize}
{We call this {\it scalarwise} $C^\infty$.}
\item Any Mackey-Cauchy-sequence (i.e.,  $t_{nm}(x_n-x_m)\to 0$  
for some $t_{nm}\to \infty$ in $\mathbb R$) converges in $E$. 
This is visibly a mild completeness requirement.
\item If $B$ is bounded closed absolutely convex, then $E_B$ 
is a Banach space.
\item If $f:\mathbb R\to E$ is scalarwise $\on{Lip}^k$, then $f$ is 
$\on{Lip}^k$, for $k>1$.
\item If $f:\mathbb R\to E$ is scalarwise $C^\infty$ then $f$ is 
differentiable at 0.
\end{enumerate}

Here a mapping $f:\mathbb R\to E$ is called $\on{Lip}^k$ if all 
derivatives up to order $k$ exist and are Lipschitz, locally on 
$\mathbb R$. That $f$ is scalarwise $C^\infty$ means $\la\o f$ is $C^\infty$  
for all continuous (equiv., bounded) linear functionals on $E$.

\subsection{Smooth mappings} 

Let $E$, and $F$ be convenient vector spaces, 
and let $U\subset E$ be $c^\infty$-open. 
A mapping $f:U\to F$ is called {\it smooth} or 
$C^\infty$, if $f\o c\in C^\infty(\mathbb R,F)$ for all 
$c\in C^\infty(\mathbb R,U)$. See \cite[3.11]{KrieglMichor97}.

If $E$ is a Fr\'echet space, then this notion coincides with all other reasonable notions of 
$C^\infty$-mappings; see below. 
Beyond Fr\'echet spaces, as a rule, there are more smooth mappings in the 
convenient setting than in other settings, e.g., $C^\infty_c$.

\subsection{Main properties of smooth calculus}
\begin{enumerate}
\item For maps on Fr\'echet spaces this 
coincides with all other reasonable definitions. On 
$\mathbb R^2$ this is non-trivial; see \cite{Boman67} or \cite[3.4]{KrieglMichor97}.
\item Multilinear mappings are smooth iff they are 
bounded; see \cite[5.5]{KrieglMichor97}.
\item If $E\supseteq U\East{f}{} F$ is smooth then the derivative 
$df:U\x E\to F$ is  
smooth, and also $df:U\to L(E,F)$ is smooth where $L(E,F)$ 
denotes the space of all bounded linear mappings with the 
topology of uniform convergence on bounded subsets; see \cite[3.18]{KrieglMichor97}.
\item The chain rule holds; see \cite[3.18]{KrieglMichor97}.
\item The space $C^\infty(U,F)$ is again a convenient vector space 
where the structure is given by the  injection
$$
C^\infty(U,F) \East{C^\infty(c,\ell)}{} 
\prod_{c\in C^\infty(\mathbb R,U), \ell\in F^*} 
C^\infty(\mathbb R,\mathbb R),
\quad f\mapsto (\ell\o f\o c)_{c,\ell},
$$
and where $C^\infty(\mathbb R,\mathbb R)$ carries the topology of compact 
convergence in each derivative separately; see \cite[3.11 and 3.7]{KrieglMichor97}.
\item The exponential law holds; see \cite[3.12]{KrieglMichor97}.: For $c^\infty$-open $V\subset F$, 
$$
C^\infty(U,C^\infty(V,G)) \cong C^\infty(U\x V, G)
$$
is a linear diffeomorphism of convenient vector spaces.
{\em Note that this (for $U=\mathbb R$) is the main assumption of variational calculus. Here it is a theorem.}
\item[7.] A linear mapping $f:E\to C^\infty(V,G)$ is smooth (by \thetag{2} equivalent to bounded) if 
and only if $E \East{f}{} C^\infty(V,G) \East{\on{ev}_v}{} G$ is smooth 
for each $v\in V$. 
({\it Smooth uniform boundedness theorem};
see \cite[theorem 5.26]{KrieglMichor97}.)
\item A mapping $f:U\to L(F,G)$ is smooth if and only if $U \East{f}{} L(F,G) \East{\on{ev}_v}{} G$ 
       is smooth for each $v\in F$, because then it is scalarwise smooth by the classical uniform 
       boundedness theorem. 
\item The following canonical mappings are smooth. This follows from the exponential law by simple 
       categorical reasoning; see \cite[3.13]{KrieglMichor97}.
\begin{align*}
&\operatorname{ev}: C^\infty(E,F)\x E\to F,\quad 
\operatorname{ev}(f,x) = f(x)\\
&\operatorname{ins}: E\to C^\infty(F,E\x F),\quad
\operatorname{ins}(x)(y) = (x,y)\\
&(\quad)^\wedge :C^\infty(E,C^\infty(F,G))\to C^\infty(E\x F,G)\\
&(\quad)^\vee :C^\infty(E\x F,G)\to C^\infty(E,C^\infty(F,G))\\
&\operatorname{comp}:C^\infty(F,G)\x C^\infty(E,F)\to C^\infty(E,G)\\
&C^\infty(\quad,\quad):C^\infty(F,F_1)\x C^\infty(E_1,E)\to 
\\&\qquad\qquad\qquad\qquad
\to C^\infty(C^\infty(E,F),C^\infty(E_1,F_1))\\
&\qquad (f,g)\mapsto(h\mapsto f\o h\o g)\\
&\prod:\prod C^\infty(E_i,F_i)\to C^\infty(\prod E_i,\prod F_i)
\end{align*}
\end{enumerate}

This ends our review of the standard results of convenient calculus. Just for the curious reader 
and to give a flavor of the arguments, 
we enclose a lemma that is used many times in the proofs of the results above.

\begin{lem*} {\rm (Special curve lemma, \cite[2.8]{KrieglMichor97})}
Let $E$ be a  locally convex vector space.
Let $x_n$ be a 
sequence which converges fast
to $x$ in $E$; i.e., for each 
$k\in \mathbb N$ the sequence $n^k(x_n-x)$ is bounded. 
Then the {\it infinite polygon} through the $x_n$ can be 
parameterized as a  
smooth curve $c:\mathbb R\to E$ such that $c(\frac1n)=x_n$ and $c(0)=x$.
\end{lem*}

\subsection{Remark}

Convenient calculus (i.e., having properties 6 and 7) exists also for:
\begin{itemize}
\item  Real analytic mappings; see \cite{KrieglMichor90} or \cite[Chapter II]{KrieglMichor97}.
\item  Holomorphic mappings; see \cite{KrieglNel85} or \cite[Chapter II]{KrieglMichor97} (using the notion of 
\cite{Fantappie30,Fantappie33}).
\item  Many classes of Denjoy Carleman ultradifferentiable functions, both of Beurling type and of 
Roumieu-type, see \cite{KMRc,KMRq,KMR15c,KMR15d}.
\item With some adaptations, $\on{Lip}^k$; see \cite{FrolicherKriegl88}. 
\item With more adaptations, even $C^{k,\al}$ 
(the $k$-th derivative is H\"older-continuous with index $\alpha$); see 
\cite{FF89}, \cite{Faure91}.
\end{itemize}

The following result is very useful if one wants to apply convenient calculus to spaces which are 
not tied to its categorical origin, like the Schwartz spaces $\mathcal S$, $\mathcal D$, or Sobolev 
spaces; for its uses see \cite{MichorMumford2013z} and \cite{KMR14}.

\begin{thm}\label{FK}
{\rm \cite[theorem 4.1.19]{FrolicherKriegl88}}
Let $c:\mathbb R\to E$ be a curve in a convenient vector space $E$. Let 
$\mathcal{V}\subset E'$ be a subset of bounded linear functionals such that 
the bornology of $E$ has a basis of $\sigma(E,\mathcal{V})$-closed sets. 
Then the following are equivalent:
\begin{enumerate}
\item $c$ is smooth
\item There exist locally bounded curves $c^{k}:\mathbb R\to E$ such that
      $\ell\o c$ is smooth $\mathbb R\to \mathbb R$ with $(\ell\o c)^{(k)}=\ell\o
      c^{k}$, for each $\ell\in\mathcal V$. 
\end{enumerate}
\end{thm}
If $E$ is reflexive, then for any point separating subset
$\mathcal{V}\subset E'$ the bornology of $E$ has a basis of 
$\si(E,\mathcal{V})$-closed subsets, by \cite[4.1.23]{FrolicherKriegl88}.

\section{Manifolds of mappings and regular Lie groups}\label{mf}

In this section I hope to demonstrate how convenient calculus allows for very short and transparent 
proofs of the core results in the theory of manifolds of smooth mappings. 

\subsection{The manifold structure on $C^\infty(M,N)$}
Let $M$ be a compact (for simplicity's sake) finite dimensional manifold and $N$ a manifold. We use an 
auxiliary Riemann metric $\bar g$ on $N$. Then 
$$\xymatrix@R=.7em{
& 0_N \ar@{_{(}->}[d] \ar@(l,dl)[dl]_-{\text{zero section\quad }} 
& & N \ar@{^{(}->}[d] \ar@(r,dr)[rd]^-{\text{ diagonal}} & 
\\
TN &  V^N \ar@{_{(}->}[l]^{\text{ open  }} \ar[rr]^{ (\pi_N,\exp^{\bar g}) }_{\cong} & & V^{N\x N} 
\ar@{^{(}->}[r]_{\text{ open  }} & N\x N
}$$
$C^\infty(M,N)$, the space of smooth mappings $M\to N$, has the following manifold structure.
A chart, centered at $f\in C^\infty(M,N)$, is:
\begin{align*}
C^\infty(M,N)\supset U_f &=\{g: (f,g)(M)\subset V^{N\x N}\} \East{u_f}{} \tilde U_f 
\subset \Ga(f^*TN)
\\
u_f(g) = (\pi_N,&\exp^{\bar g})\i \o (f,g),\quad u_f(g)(x) = (\exp^{\bar g}_{f(x)})\i(g(x)) 
\\
(u_f)\i(s) &= \exp^{\bar g}_f\o s, \qquad (u_f)\i(s)(x) = \exp^{\bar g}_{f(x)}(s(x))
\end{align*}

\begin{lem} \label{firstlemma} $C^\infty(\mathbb R,\Ga(M;f^*TN)) = \Ga(\mathbb R\x M; \on{pr_2}^* f^*TN)$
\end{lem}

This follows by cartesian closedness after trivializing the bundle $f^*TN$.

\begin{lem} The chart changes are smooth ($C^\infty$) 

$$
\tilde U_{f_1}\ni s\mapsto (\pi_N,\exp^{\bar g})\o s 
\mapsto (\pi_N,\exp^{\bar g})\i\o(f_2,\exp^{\bar g}_{f_1}\o s)
$$
\end{lem}

Since they map smooth curves to smooth curves.
                                                        
\begin{lem} $C^\infty(\mathbb R,C^\infty(M,N))\cong C^\infty(\mathbb R\x M,N)$.
\end{lem}

By  lemma \ref{firstlemma}.

\begin{lem} Composition $C^\infty(P,M)\x C^\infty(M,N)\to C^\infty(P,N)$, $(f,g)\mapsto g\o f$, is 
smooth
\end{lem}

Since it maps smooth curves to smooth curves.

\begin{coro} The tangent bundle of $C^\infty(M,N)$ is given by 
$$
TC^\infty(M,N)= C^\infty(M,TN) \East{C^\infty(M,\pi_N)}{} C^\infty(M,N).
$$
\end{coro}

This follows from the chart structure.

\subsection{Regular Lie groups}
We consider a  smooth Lie group $G$ 
with Lie algebra $\mathfrak g=T_eG$ modelled on convenient 
vector spaces. 
The notion of a regular Lie group is originally due to 
\cite{OmoriMaedaYoshioka80,OmoriMaedaYoshioka81,OmoriMaedaYoshioka81b,OmoriMaedaYoshioka82,OmoriMaedaYoshioka83,OmoriMaedaYoshiokaKobayashi83}
for Fr\'echet Lie groups, was 
weakened and made more transparent by \cite{Milnor84}, and then carried over to convenient Lie 
groups in \cite{KM97r}, see also \cite[38.4]{KrieglMichor97}.
We shall write $\mu:G\x G\to G$ for the mutiplication with $x.y = \mu(x,y) = \mu_x(y) = \mu^y(x)$ 
for left and right translation.

A Lie group $G$
is called {\em regular}  if the following holds:
\begin{itemize}
\item 
For each smooth curve 
$X\in C^{\infty}(\mathbb R,\mathfrak g)$ there exists a curve 
$g\in C^{\infty}(\mathbb R,G)$ whose right logarithmic derivative is $X$, i.e.,
$$
\begin{cases} g(0) &= e \\
\p_t g(t) &= T_e(\mu^{g(t)})X(t) = X(t).g(t)
\end{cases} 
$$
The curve $g$ is uniquely determined by its initial value $g(0)$, if it
exists.
\item
Put $\on{evol}^r_G(X)=g(1)$ where $g$ is the unique solution required above. 
Then $\on{evol}^r_G: C^{\infty}(\mathbb R,\mathfrak g)\to G$ is required to be
$C^{\infty}$ also. We have $\on{Evol}^X_t:= g(t) =\on{evol}_G(tX)$.
\end{itemize}
Up to now, every Lie group modeled on convenient vector spaces is regular.

\begin{thm} For each compact manifold $M$, the diffeomorphism group $\Diff(M)$ is a regular Lie group.
\end{thm}

\begin{proof} $\Diff(M)\East{open}{} C^\infty(M,M)$. Composition is smooth by restriction.
Inversion is smooth: If $t\mapsto f(t,\quad)$ is a smooth curve in $\Diff(M)$, then 
$f(t,\quad)\i$ satisfies the implicit equation $f(t,f(t,\quad)\i(x))=x$, so by the finite 
dimensional implicit function theorem, $(t,x)\mapsto f(t,\quad)\i(x)$ is smooth. So inversion maps 
smooth curves to smooth curves, and is smooth.

Let $X(t,x)$ be a time dependent vector field on $M$ (in $C^\infty(\mathbb R,\X(M))$).
Then $\on{Fl}^{\p_t\x X}_s(t,x)=(t+s,\on{Evol}^X(t,x))$ satisfies the ordinary differential 
equation
$$
\quad\p_t\on{Evol}(t,x) = X(t,\on{Evol}(t,x)).
$$
If $X(s,t,x)\in C^\infty(\mathbb R^2,\X(M))$ is a smooth curve of smooth curves in $\X(M)$,
then obviously the solution of the equation depends smoothly also on the further variable $s$,
thus $\on{evol}$ maps smooth curves of time dependant vector fields to smooth curves of 
diffeomorphism. 
\end{proof}

\subsubsection*{\bf Groups of smooth diffeomorphisms on $\mathbb R^n$}
If we consider the group of all orientation preserving 
diffeomorphisms $\Diff(\mathbb R^n)$ of $\mathbb R^n$, it is not an open subset of 
$C^\infty(\mathbb R^n,\mathbb R^n)$ with the compact $C^\infty$-topology. 
So it is not a smooth manifold in the usual sense, but we may consider it as a Lie group in the 
cartesian closed category of  Fr\"olicher spaces, see \cite[Section 23]{KrieglMichor97}, with the structure 
induced by the injection 
$f\mapsto (f,f\i)\in C^\infty(\mathbb R^n,\mathbb R^n)\x C^\infty(\mathbb R^n,\mathbb R^n)$.
Or one can use the setting of `manifolds' based on smooth curves instead of charts, with lots of 
extra structure (tangent bundle, parallel transport, geodesic structure), described in 
\cite{Michor84a,Michor84b}; this gives a category of smooth `manifolds' where 
those which have Banach spaces as tangent fibes are exactly the usual smooth manifolds modeled on 
Banach spaces, which is cartesian closed: $C^\infty(M,N)$ and $\Diff(M)$ are always `manifolds' for 
`manifolds' $M$ and $N$, and the exponential law holds. 

We shall now describe regular Lie groups in $\Diff(\mathbb R^n)$ 
which are given by 
diffeomorphisms of the form $f = \on{Id}_{\mathbb R} + g$ 
where $g$ is in some specific convenient vector space of bounded functions in 
$C^\infty(\mathbb R^n,\mathbb R^n)$.
Now we discuss these spaces on $\mathbb R^n$, we describe the smooth curves in them, and we describe the 
corresponding groups. These results are from \cite{MichorMumford2013z} and from \cite{KMR14,KMR15d} for the 
more exotic groups.

\subsubsection*{The group $\Diff_{\mathcal B}(\mathbb R^n)$}
The space $\mathcal B(\mathbb R^n)$ (called $\mathcal D_{L^\infty}(\mathbb R^n)$ by  
\cite{Schwartz66}) consists of all smooth functions which have all derivatives (separately) bounded.
It is a Fr\'echet space. 
By \cite{Vogt83}, the space $\mathcal B(\mathbb R^n)$ is linearly 
isomorphic to $\ell^\infty\hat\otimes\, \mathfrak s$ for any completed tensor-product between the 
projective one and the injective one, where $\mathfrak s$ is the nuclear Fr\'echet space of rapidly 
decreasing real sequences. Thus $\mathcal B(\mathbb R^n)$ is not reflexive, not nuclear, not 
smoothly paracompact.
\newline
{\it The space $C^\infty(\mathbb R,\mathcal{B}(\mathbb R^n))$ of smooth
curves in $\mathcal{B}(\mathbb R^n)$ consists of all functions 
$c\in C^\infty(\mathbb R^{n+1},\mathbb R)$ satisfying the following 
property:
\begin{enumerate}
\item[$\bullet$]
For all $k\in \mathbb N_{\ge0}$, $\al\in \mathbb N_{\ge0}^n$ and each $t\in \mathbb R$ the expression
$\p_t^{k}\p^\al_x c(t,x)$  is uniformly bounded in $x\in \mathbb R^n$, locally
in $t$. 
\end{enumerate} 
}
To see this use Theorem \ref{FK} for the set $\{\on{ev}_x: x\in \mathbb R\}$ of point
evaluations in $\mathcal{B}(\mathbb R^n)$. 
Here $\p^\al_x = \frac{\p^{|\al|}}{\p x^\al}$ and $c^k(t)=\p_t^kf(t,\quad)$.
\newline
{\it $\Diff^+_{\mathcal B}(\mathbb R^n)=\bigl\{f=\on{Id}+g: g\in\mathcal B(\mathbb R^n)^n, 
    \det(\mathbb I_n + dg)\ge \ep > 0 \bigr\}$ denotes the corresponding group}, 
                see below.

\subsubsection*{The group $\Diff_{W^{\infty,p}}(\mathbb R^n)$}
For $1\le p <\infty$, the space 
$W^{\infty,p}(\mathbb R^n)=\bigcap_{k\ge 1}L^p_k(\mathbb R^n)$ is the intersection of all $L^p$-Sobolev 
spaces, the space of all smooth functions such that each partial derivative is in $L^p$. It is a reflexive Fr\'echet space.
It is called $\mathcal D_{L^p}(\mathbb R^n)$ in \cite{Schwartz66}.
By \cite{Vogt83}, the space $W^{\infty,p}(\mathbb R^n)$ is linearly isomorphic to 
$\ell^p\hat\otimes\, \mathfrak s$. Thus it is not nuclear, not Schwartz, not Montel, and smoothly 
paracompact only if $p$ is an even integer.
\newline
{\it The space $C^\infty(\mathbb R,H^\infty(\mathbb R^n))$ of smooth
curves in $W^{\infty,p}(\mathbb R^n)$ consists of all functions 
$c\in C^\infty(\mathbb R^{n+1},\mathbb R)$ satisfying the following 
property:
\begin{enumerate}
\item[$\bullet$]
For all $k\in \mathbb N_{\ge0}$, $\al\in \mathbb N_{\ge0}^n$  the expression 
$\|\p_t^{k}\p^\al_xf(t,\quad)\|_{L^p(\mathbb R^n)}$ is locally bounded  
near each $t\in \mathbb R$. 
\end{enumerate} 
}
The proof is literally the same as for $\mathcal B(\mathbb R^n)$, noting that the point evaluations are 
continuous on each Sobolev space $L^p_k$ with $k>\frac np$.
\newline
{\it $\Diff^+_{W^{\infty,p}}(\mathbb R^n)=\bigl\{f=\on{Id}+g: g\in W^{\infty,p}(\mathbb R^n)^n, 
                \det(\mathbb I_n + dg)>0\bigr\}$ 
                denotes the corresponding group}.

\subsubsection*{The group $\Diff_{\mathcal S}(\mathbb R^n)$}
The algebra $\mathcal S(R^n)$  of rapidly decreasing functions is a reflexive nuclear Fr\'echet space.
\newline
{\it The space $C^\infty(\mathbb R,\mathcal S(\mathbb R^n))$ of smooth
curves in $\mathcal S(\mathbb R^n)$ consists of all functions 
$c\in C^\infty(\mathbb R^{n+1},\mathbb R)$ satisfying the following 
property:
\begin{enumerate}
\item[$\bullet$]
For all $k,m\in \mathbb N_{\ge0}$ and $\al\in \mathbb N_{\ge0}^n$,  the expression
$(1+|x|^2)^m\p_t^{k}\p^\al_xc(t,x)$ is uniformly bounded in $x\in \mathbb R^n$, locally uniformly bounded  
in $t\in \mathbb R$.
\end{enumerate} 
}
{\it $\Diff^+_{\mathcal S}(\mathbb R^n)=\bigl\{f=\on{Id}+g: g\in \mathcal S(\mathbb R^n)^n, 
                \det(\mathbb I_n + dg)>0\bigr\}$ is the corresponding group.} 

\subsubsection*{The group  $\Diff_{c}(\mathbb R^n)$}
The algebra $C^\infty_c(\mathbb R^n)$ of all smooth functions with compact support is a nuclear 
(LF)-space.
\newline 
{\it The space $C^\infty(\mathbb R,C^\infty_c(\mathbb R^n))$ of smooth
curves in $C^\infty_c(\mathbb R^n)$ consists of all functions 
$f\in C^\infty(\mathbb R^{n+1},\mathbb R)$ satisfying the following 
property:
\begin{enumerate}
\item[$\bullet$]
For 
each compact interval $[a,b]$ in $\mathbb R$ there exists a compact subset $K\subset \mathbb R^n$
such that $f(t,x)=0$ for  $(t,x)\in [a,b]\x (\mathbb R^n\setminus K)$.
\end{enumerate} 
}
{\it $\Diff_c(\mathbb R^n)=\bigl\{f=\on{Id}+g: g\in C^\infty_c(\mathbb R^n)^n, 
                \det(\mathbb I_n + dg)>0\bigr\}$ is the correponding group.} 
    The case $\Diff_c(\mathbb R^n)$ is well-known since 1980. 

\subsubsection*{Ideal properties of function spaces}

The function spaces discussed are boundedly mapped into each other as follows:
$$\xymatrix{
C^\infty_c(\mathbb R^n) \ar[r]  & \mathcal S(\mathbb R^n)  \ar[r] & W^{\infty,p}(\mathbb R^n) 
\ar[r]^{p\le q} & W^{\infty,q}(\mathbb R^n) \ar[r]
 &  \mathcal B(\mathbb R^n)
}$$
and each space is a bounded locally convex algebra and a bounded $\mathcal B(\mathbb R^n)$-module.
Thus each space is an ideal in each larger space. 

\begin{thm}[\cite{MichorMumford2013z} and \cite{KMR14}]
The sets of diffeomorphisms $\Diff_c(\mathbb R^n)$, 
$\Diff_{\mathcal S}(\mathbb R^n)$,
$\Diff_{H^\infty}(\mathbb R^n)$, and 
$\Diff_{\mathcal B}(\mathbb R^n)$
are all smooth regular Lie groups.
We have the following smooth injective group homomorphisms
$$\xymatrix{
\Diff_c(\mathbb R^n) \ar[r] & \Diff_{\mathcal S}(\mathbb R^n) \ar[r] & \Diff_{W^{\infty,p}}(\mathbb R^n) \ar[r]  
&  \Diff_{\mathcal B}(\mathbb R^n)
}.$$
Each group is a normal subgroup in any other in which it is contained, in particular in 
$\on{Diff}_{\mathcal B}(\mathbb R^n)$.
\end{thm}

The proof of this theorem relies on repeated use of the Fa\`a~di~Bruno formula for higher 
derivatives of composed functions. This offers difficulties on non-compact manifolds, where one 
would need a non-commutative Fa\`a~di~Bruno formula for iterated covariant derivatives.
In the paper \cite{KMR14} many more similar groups are discussed, modeled on spaces of 
Denjoy-Carleman ultradifferentiable functions. It is also showm that for $p>1$ the group
$\Diff_{W^{\infty,p}\cap L^1}(\mathbb R^n)$ is only a topological group with smooth 
right translations --- a property which is similar to the one of finite order Sobolev groups 
$\Diff_{W^{k,p}}(\mathbb R^n)$. Some of these groups were used extensively in 
\cite{MumfordMichor13}. 

\begin{coro}
$\Diff_{\mathcal B}(\mathbb R^n)$ acts on $\Ga_c$, $\Ga_{\mathcal S}$ and $\Ga_{H^\infty}$ of any 
tensorbundle over $\mathbb R^n$ by pullback. The infinitesimal action of the Lie algebra $\X_{\mathcal B}(\mathbb R^n)$ on 
these spaces by the Lie derivative thus maps each of these spaces into itself. 
A fortiori, $\Diff_{H^\infty}(\mathbb R^n)$ acts on $\Ga_{\mathcal S}$ of any tensor bundle by pullback. 
\end{coro}

\subsection{Manifolds of immersions and shape spaces}

For finite dimensional manifolds $M$, $N$ with $M$ compact,
$\Emb(M,N)$, the space of embeddings of $M$ into $N$, is open in $C^\infty(M,N)$, so it is a smooth 
manifold. $\Diff(M)$ acts freely and smoothly from the right on $\Emb(M,N)$.

\begin{thm} $\Emb(M,N)\to \Emb(M,N)/\Diff(M) = B(M,N)$ is a principal fiber bundle with structure group 
$\Diff(M)$.
\end{thm}

Note that $B(M,N)$ is the smooth  manifold of all submanifolds of $N$ which are of diffeomorphism 
type $M$. Therefore it is also called the {\it nonlinear Grassmannian} in \cite{G-BV14}, where this 
theorem is extended to the case when $M$ has boundary. From another point of view, $B(M,N)$ is 
called the {\it differentiable Chow variety} in \cite{Michor127}. It is an example of a {\it shape 
space}.

\begin{proof} We use an auxiliary Riemannian metric $\bar g$ on $N$. Given $f\in \Emb(M,N)$, we 
view $f(M)$ as a submanifold of $N$ and we split the the tangent bundle of $N$ along $f(M)$ as 
$TN|_{f(M)}= \on{Nor}(f(M))\oplus Tf(M)$. The exponential mapping describes a tubular neighborhood 
of $f(M)$ via
$$
\on{Nor}(f(M))\East{\exp^{\bar g}}{\cong} W_{f(M)}\East{p_{f(M)}}{} f(M).
$$
If $g:M\to N$ is $C^1$-near to $f$, then $\ph(g):=f\i\o p_{f(M)}\o g\in \Diff(M)$ and we may 
consider 
$g\o \ph(g)\i\in \Ga(f^*W_{f(M)}) \subset \Ga(f^*\on{Nor}(f(M)))$. 
This is the required local splitting.
\end{proof}

$\Imm(M,N)$, the space of immersions $M\to N$, is open in $C^\infty(M,N)$, and is thus a smooth 
manifold. The regular Lie group $\Diff(M)$ acts smoothly from the right, but no longer freely. 

\begin{thm}[\cite{CerveraMascaroMichor91}]
For an immersion $f:M\to N$, the isotropy group 
$$
\Diff(M)_f=\{\ph\in\Diff(M): f\o \phi= f\}
$$
is always a finite group, acting freely on $M$; so $M\East{p}{} M/\Diff(M)_f$ is a 
covering of manifolds and 
$f$ factors to $f=\bar f\o p$. 

Thus $\Imm(M,N)\to \Imm(M,N)/\Diff(M) =: B_i(M,N)$ is a projection onto an infinite dimensional 
orbifold.
\end{thm}

The space $B_i(M,N)$ is another example of a {\it shape space}. It appeared in the form of 
$B_i(S^1,\mathbb R^2)$, the shape space of plane immersed curves, in \cite{Michor2006c} and 
\cite{Michor2007}.

\section{Weak Riemannian manifolds}

If an infinite dimensional manifold is not modeled on a Hilbert space, then a Riemannian metric 
cannot describe the topology on each tangent space. We have to deal with more complicated 
situations.

\subsection{Manifolds, vector fields, differential forms}
Let $M$ be a smooth manifold modeled on convenient vector spaces. Tangent vectors to 
$M$ are kinematic ones. 

The reason for this is that eventually we want flows of vector 
fields, and that there are too many derivations in 
infinite dimensions, even on a Hilbert space $H$: Let $\al\in L(H,H)$ be a continuous linear 
functional which vanishes on the subspace of compact operators, thus also on $H\otimes H$. Then 
$f\mapsto \al(d^2f(0))$ is a derivation at 0 on $C^\infty(H)$, since 
$\al(d^2(f.g)(0))=\al\big(d^2f(0).g(0)+df(0)\otimes dg(0) + dg(0)\otimes df(0) + f(0).d^2g(0)\big)$
and $\al$ vanishes on the two middle terms. There are even non-zero derivations which differentiate 3 times, 
see \cite[28.4]{KrieglMichor97}. 

The (kinematic) tangent bundle $TM$ is then a smooth vector bundle as usual. 
Differential forms of degree $k$ are then smooth sections of the bundle 
$L^k_{\text{skew}}(TM;\mathbb R)$ of skew symmetric 
$k$-linear functionals on the tangent bundle, since this is the only version which admits exterior 
derivative, Lie derivatives along vector field, and pullbacks along arbitray smooth mappings; 
see \cite[33.21]{KrieglMichor97}. The de~Rham cohomology equals singular cohomology with real 
coefficients if the manifold is smoothly paracompact; see \cite [Section 34]{KrieglMichor97}.
If a vector field admits a flow, then each integral curve is 
uniquely given as a flow line; see \cite[32.14]{KrieglMichor97}.

\subsection{Weak Riemannian manifolds}
Let  $(M,g)$  be  a smooth manifold
modeled on  convenient locally convex vector spaces. A smooth Riemannian metric on $M$ is called weak if 
$g_x:T_xM\to  T_x^*M$ is only injective for each $x\in M$. 
The  image $g(TM)\subset  T^*M$  is  called  the {\em smooth cotangent
bundle}  associated  to  $g$.  
Then  $g\i$ is the metric on the
smooth  cotangent bundle as well as the morphism $g(TM)\to TM$.
We have a special class of 1-forms
$\Om_g^1(M):=  \Ga(g(TM))$ for which the musical mappings makes sense: $\al^\sharp=g\i\al\in\X(M)$ 
and $X^\flat  =  gX$. These 1-forms separate points on $TM$.
The exterior derivative is 
defined by
$d:\Om_g^1(M) \to \Om^2(M)=\Ga(L^2_{\text{skew}}(TM;\mathbb R))$
since the embedding $g(TM)\subset T^*M$ is a smooth fiber linear mapping. 

Existence of the Levi-Civita covariant derivative is equivalent to:
{\em The metric itself admits {\it symmetric} gradients with
respect to itself.}
Locally this means:
If $M$ is $c^\infty$-open in a
convenient vector space $V_M$. Then:
$$
D_{x,X}g_x(X,Y) = g_x(X,\on{grad}_1g(x)(X,Y)) 
= g_x(\on{grad}_2g(x)(X,X),Y)
$$
where $D_{x,X}$ denote the directional derivative at $x$ in the direction $X$, and where
the mappings $\on{grad}_1g$ and $\on{sym}\on{grad}_2g: M\x V_M\x V_M\to V_M$, given by
$(x,X)\mapsto \on{grad}_{1,2}g(x)(X,X)$, are smooth and quadratic in $X\in V_M$.

\subsection{Weak Riemannian metrics on spaces of immersions}
For a compact manifold $M$ and a finite dimensional Riemannian manifold $(N,\bar g)$ we can consider the 
following weak Riemannian metrics on the manifold $\Imm(M,N)$ of smooth immersions $M\to N$:
\begin{align*}
G^0_f(h,k) &= \int_M \bar g(h,k) \vol(f^*\bar g) &\quad&\text{the  }L^2\text{-metric,}
\\
G^s_f(h,k) &= \int_M \bar g((1+\De^{f^*\bar g})^s h,k) \vol(f^*\bar g) &\quad&\text{a Sobolev 
metric of order }s,
\\
G^\ph_f(h,g) &= \int_M \Ph(f) \bar g(h,k) \vol(f^*\bar g)  & &\text{an almost local metric.}
\end{align*}
Here $\vol(f^*\bar g)$ is the volume density on $M$ of the pullback metric $g=f^*\bar g$, 
where $\De^{f^*\bar g}$ is the (Bochner) Laplacion with respect to $g$ and $\bar g$ acting on 
sections of $f^*TN$,
and where $\Ph(f)$ is a positive function of the total volume $\on{Vol}(f^*g)=\int_M \vol(f^*g)$, 
of the scalar curvature $\on{Scal}(f^*\bar g)$, and of the mean curvature $\on{Tr}(S^f)$, $S^f$ 
being the second fundamental form. See \cite{Bauer2011b}, \cite{Bauer2012a} for more information.
All these metrics are invariant for the right action of the reparameterization group $\Diff(M)$, so 
they descend to metrics on shape space $B_i(M,N)$ (off the singularities) such that the 
projection $\Imm(M,N)\to B_i(M,N)$ is a Riemannian submersion of a benign type: the $G$-orthogonal 
component to the tangent space to the $\Diff(M)$-orbit consists always of smooth vector fields. 
So there is no need to use the notion of robust weak Riemannian metrics discussed below. 

\begin{thm}
1. Geodesic distance on $\Imm(M,N)$, defined as the infimum of path-lenghts of smooth isotopies between two 
immersions, vanishes for the $L^2$-metric $G^0$.

2. Geodesic distance is positive on $B_i(M,N)$ for the almost local metric $G^\Ph$ if 
$\Ph(f)\ge 1+ A\on{Tr}(S^F)$, or if $\Ph(f)\ge A\on{Vol}(f^*\bar g)$, for some $A>0$.

3. Geodesic distance is positive on $B_i(M,N)$ for the Sobolev metric $G^s$ if $s\ge 1$.

4. The geodesic equation is locally well-posed on $\Imm(M,N)$ for the Sobolev metric $G^s$ if 
$s\ge 1$, and globally well-posed (and thus geodesically complete) on $\Imm(S^1,\mathbb R^n)$, if 
$s\ge 2$.
\end{thm}

{1} is due to  \cite{Michor2006c} for $B_i(S^1,\mathbb R^2)$, to \cite{Michor2005} for $B_i(M,N)$ and for 
$\Diff(M)$, which combines to the result for $\Imm(M,N)$ as noted in \cite{Bauer2012c}. {2} is proved 
in \cite{Bauer2012a}. For {3} see \cite{Bauer2011b}.
{4} is due to \cite{Bruveris2014} and \cite{Bruveris2014b}.

\subsection{Weak right invariant Riemannian metrics on regular Lie groups}\label{wRmLie}
Let $G$ be a regular convenient Lie group, with Lie algebra $\mathfrak g$. 
Let $\mu:G\x G\to G$ be the group multiplication, $\mu_x$ the left
translation and $\mu^y$ the right translation, 
$\mu_x(y)=\mu^y(x)=xy=\mu(x,y)$. 
Let $L,R:\mathfrak g\to \X(G)$ be the left- and right-invariant vector field mappings, given by 
$L_X(g)=T_e(\mu_g).X$ and $R_X=T_e(\mu^g).X$, respectively.
They are related by $L_X(g)=R_{\on{Ad}(g)X}(g)$.
Their flows are given by 
\begin{align*}
\on{Fl}^{L_X}_t(g)&= g.\exp(tX)=\mu^{\exp(tX)}(g),
\\
\on{Fl}^{R_X}_t(g)&= \exp(tX).g=\mu_{\exp(tX)}(g).
\end{align*}
The right
Maurer--Cartan form $\ka=\ka^r\in\Om^1(G,\mathfrak g)$ is given by  
$\ka_x(\xi):=T_x(\mu^{x\i})\cdot \xi$. 
\newline
The left
Maurer--Cartan form $\ka^l\in\Om^1(G,\mathfrak g)$ is given by  
$\ka_x(\xi):=T_x(\mu_{x\i})\cdot \xi$. 
\newline
{\it $\ka^r$ satisfies the left Maurer-Cartan equation 
$d\ka-\tfrac12[\ka,\ka]^\wedge_{\mathfrak g}=0$}, 
where
$[\quad,\quad]_\wedge$ denotes the wedge product of $\mathfrak g$-valued forms on
$G$ induced by the Lie bracket. Note that
$\tfrac12[\ka,\ka]_\wedge (\xi,\et) = [\ka(\xi),\ka(\et)]$.
\newline
{\it $\ka^l$ satisfies the right Maurer-Cartan equation 
  $d\ka+\tfrac12[\ka,\ka]^\wedge_{\mathfrak g}=0$}. 

Namely, evaluate $d\ka^r$ on right invariant vector fields $R_X,R_Y$ for $X,Y\in{\mathfrak g}$.
\begin{multline*}
(d\ka^r)(R_X,R_Y) = R_X(\ka^r(R_Y)) - R_Y(\ka^r(R_X)) - \ka^r([R_X,R_Y])
\\
= R_X(Y) - R_Y(X) + [X,Y] = 0-0 +[\ka^r(R_X),\ka^r(R_Y)].
\end{multline*}

The (exterior) derivative of the function $\on{Ad}:G\to GL(\mathfrak g)$ can be 
expressed by
\begin{displaymath}
d\on{Ad} = \on{Ad}.(\on{ad}\o\ka^l) = (\on{ad}\o \ka^r).\on{Ad}
\end{displaymath}
since we have
$d\on{Ad}(T\mu_g.X) = \frac d{dt}|_0 \on{Ad}(g.\exp(tX))
= \on{Ad}(g).\on{ad}(\ka^l(T\mu_g.X))$.

Let $\ga=\mathfrak g\x\mathfrak g\to\mathbb R$ be a positive-definite bounded (weak) inner product. Then  
\begin{equation*}
\ga_x(\xi,\et)=\ga\big( T(\mu^{x\i})\cdot\xi,\, T(\mu^{x\i})\cdot\et\big) =  
     \ga\big(\ka(\xi),\,\ka(\et)\big)
\end{equation*}
is a right-invariant (weak) Riemannian metric on $G$
and any (weak) right-invariant bounded Riemannian metric is of this form, for
suitable  $\ga$.
Denote by $\check\ga:\mathfrak g\to \mathfrak g^*$ the mapping induced by $\ga$, 
from the 
Lie algebra into its dual (of bounded linear functionals) 
and by 
$\langle \al, X \rangle_{\mathfrak g}$ the duality evaluation between $\al\in\mathfrak g^*$ and 
$X\in \mathfrak g$.
 
Let $g:[a,b]\to G$ be a smooth curve.  
The velocity field of $g$, viewed in the right trivializations, 
coincides  with the right logarithmic derivative
\begin{displaymath}
\de^r(g)=T(\mu^{g\i})\cdot \partial_t g =  
\ka(\partial_t g) = (g^*\ka)(\partial_t).
\end{displaymath}
The energy of the curve $g(t)$ is given by  
\begin{equation*}
E(g) = \frac12\int_a^b\ga_g(g',g')dt = \frac12\int_a^b 
     \ga\big( (g^*\ka)(\partial_t),(g^*\ka)(\partial_t)\big)\, dt. 
\end{equation*}
For a variation $g(s,t)$ with fixed endpoints we then use that
$$d(g^*\ka)(\p_t,\p_s)=\p_t(g^*\ka(\p_s))-\p_s(g^*\ka(\p_t))-0,$$ partial integration, and the
left Maurer--Cartan equation to obtain
\begin{align*} 
&\partial_sE(g) = \frac12\int_a^b2 
     \ga\big( \partial_s(g^*\ka)(\partial_t),\,
                            (g^*\ka)(\partial_t)\big)\, dt
\\&
= \int_a^b \ga\big( \partial_t(g^*\ka)(\partial_s) - 
         d(g^*\ka)(\partial_t,\partial_s),\,
                (g^*\ka)(\partial_t)\big)\,dt
\\&
= -\int_a^b \ga\big( (g^*\ka)(\partial_s),\,\partial_t(g^*\ka)(\partial_t)\big)\,dt 
\\&\qquad\qquad
  - \int_a^b \ga\big( [(g^*\ka)(\partial_t),(g^*\ka)(\partial_s)],\, (g^*\ka)(\partial_t)\big)\, dt
\\&
= -\int_a^b \big\langle\check\ga(\partial_t(g^*\ka)(\partial_t)),\,
  (g^*\ka)(\partial_s)\big\rangle_{\mathfrak g} \, dt
\\&\qquad\qquad
  - \int_a^b \big\langle  \check\ga((g^*\ka)(\partial_t)),\, 
  \on{ad}_{(g^*\ka)(\partial_t)}(g^*\ka)(\partial_s)\big\rangle_{\mathfrak g} \, dt
\\&
= -\int_a^b 
\big\langle\check\ga(\partial_t(g^*\ka)(\partial_t)) + (\on{ad}_{(g^*\ka)(\partial_t)})^{*}\check\ga((g^*\ka)(\partial_t)),\, 
  (g^*\ka)(\partial_s)\big\rangle_{\mathfrak g} \, dt.
\end{align*}
Thus the curve $g(0,t)$ is critical for the energy if and only if
$$
\check\ga(\partial_t(g^*\ka)(\partial_t)) + 
(\on{ad}_{(g^*\ka)(\partial_t)})^{*}\check\ga((g^*\ka)(\partial_t)) = 0.
$$
In terms of the right logarithmic derivative $u:[a,b]\to \mathfrak g$ of  
$g:[a,b]\to G$, given by  
$u(t):= g^*\ka(\partial_t) = T_{g(t)}(\mu^{g(t)\i})\cdot g'(t)$, 
the {\em geodesic equation} has the expression
\begin{equation*}
\boxed{\quad
\p_t u = - \,\check\ga\i\on{ad}(u)^{*}\;\check\ga(u)\quad} 
\end{equation*}
Thus the geodesic equation exists in general if and only if 
$\on{ad}(X)^{*}\check\ga(X)$ is in the image of $\check\ga:\mathfrak g\to\mathfrak g^*$, i.e.,
\begin{equation*}
\on{ad}(X)^{*}\check\ga(X) \in \check\ga(\mathfrak g)
\end{equation*}
 for every 
$X\in\mathfrak X$. This condition then leads to the existence of the  
Christoffel symbols. Arnold in \cite{Arnold66} asked for the more restrictive condition  
$\on{ad}(X)^{*}\check\ga(Y) \in \check \ga(\mathfrak g)$ for all $X,Y\in\mathfrak g$.
The geodesic equation for the {\it momentum} $p:=\ga(u)$ is
$$
p_t = - \on{ad}(\check\ga\i(p))^*p.
$$
There are situations, see theorem \ref{Rmap} or \cite{BBM14b}, where 
only the more general condition is satisfied, but where
the usual 
transpose $\on{ad}^\top(X)$ of $\on{ad}(X)$, 
\begin{equation*}
\on{ad}^\top(X) := \check\ga\i\o\on{ad}_X^*\o\; \check\ga
\end{equation*}
does not exist for all $X$. 

We describe now the {\em covariant derivative} and the {\em curvature}.
The right trivialization $(\pi_G,\ka^r):TG\to G\x{\mathfrak g}$  
induces the isomorphism $R:C^\infty(G,{\mathfrak g})\to \X(G)$, given by  
$R(X)(x):= R_X(x):=T_e(\mu^x)\cdot X(x)$, for $X\in C^\infty(G,{\mathfrak g})$ and
$x\in G$. Here $\X(G):=\Ga(TG)$ denotes the Lie algebra of all vector 
fields. For the Lie bracket and the Riemannian metric we have 
\begin{align*} 
[R_X,R_Y] &= R(-[X,Y]_{\mathfrak g} + dY\cdot R_X - dX\cdot R_Y),
\\
R\i[R_X,R_Y] &= -[X,Y]_{\mathfrak g} + R_X(Y) - R_Y(X),
\\
\ga_x(R_X(x),R_Y(x)) &= \ga( X(x),Y(x))\,,\, x\in G. 
\end{align*}
In what follows, we shall perform all computations in $C^\infty(G,{\mathfrak g})$ instead of 
$\X(G)$. In particular, we shall use the convention
\begin{displaymath}
\nabla_XY := R\i(\nabla_{R_X}R_Y)\quad\text{ for }X,Y\in C^\infty(G,{\mathfrak g})
\end{displaymath}
to express the Levi-Civita covariant derivative. 
 
\begin{lem}
{\rm \cite[3.3]{BBM14b}}
Assume that for all $\xi\in{\mathfrak g}$ the element $\on{ad}(\xi)^*\check\ga(\xi)\in\mathfrak g^*$ is in the image of
$\check\ga:\mathfrak g\to\mathfrak g^*$   
and that  
$\xi\mapsto \check\ga\i\on{ad}(\xi)^*\check\ga(\xi)$ is bounded  quadratic (or, equivalently, smooth). 
Then the Levi-Civita covariant derivative of the metric $\ga$ 
exists and is given for any $X,Y \in C^\infty(G,{\mathfrak g})$  in
terms of the isomorphism $R$ by
\begin{equation*}
\nabla_XY= dY.R_X + \rh(X)Y - \frac12\on{ad}(X)Y,
\end{equation*}
where 
\[
\rh(\xi)\et = \tfrac14\check\ga\i\big(\on{ad}_{\xi+\et}^*\check\ga(\xi+\et) - \on{ad}_{\xi-\et}^*\check\ga(\xi-\et)\big) = \tfrac12\check\ga\i\big(\on{ad}_\xi^*\check\ga(\et) + \on{ad}_\et^*\check\ga(\xi)\big)
\]
is the polarized version. 
$\rh:{\mathfrak g}\to L({\mathfrak g},{\mathfrak g})$ is bounded, and we have $\rh(\xi)\et=\rh(\et)\xi$.  
We also have
\begin{gather*}
\ga\big(\rh(\xi)\et,\ze\big) = \frac12\ga(\xi,\on{ad}(\et)\ze) + \frac12\ga(\et,\on{ad}(\xi)\ze),
\\
\ga(\rh(\xi)\et,\ze) + \ga(\rh(\et)\ze,\xi) + \ga(\rh(\ze)\xi,\xi) = 0.
\end{gather*}
\end{lem} 


For $X,Y\in C^\infty(G,{\mathfrak g})$ we have
$$ 
[R_X,\on{ad}(Y)] = \on{ad}(R_X(Y))\quad\text{  and }\quad  
[R_X,\rh(Y)] = \rh(R_X(Y)).
$$
The {\em Riemannian curvature} is then computed by  
\begin{align*} 
\mathcal{R}(X,Y) &=  
[\nabla_X,\nabla_Y]-\nabla_{-[X,Y]_{\mathfrak g}+R_X(Y)-R_Y(X)}
\\&
= [R_X+\rh_X-\tfrac12\on{ad}_X, 
     R_Y+\rh_Y-\tfrac12\on{ad}_Y]
\\&\quad
- R({-[X,Y]_{\mathfrak g} + R_X(Y) - R_Y(X)}) 
-\rh({-[X,Y]_{\mathfrak g} + R_X(Y)\! - R_Y(X)})
\\&\quad
+\frac12\on{ad}(-[X,Y]_{\mathfrak g} + R_X(Y)\! - R_Y(X)) 
\\&
= [\rh_X,\rh_Y] +\rh_{[X,Y]_{\mathfrak g}} -\frac12[\rh_X,\on{ad}_Y] +\frac12[\rh_Y,\on{ad}_X] 
-\frac14\on{ad}_{[X,Y]_{\mathfrak g}}.
\end{align*}
which is visibly a tensor field.

For the numerator of the sectional curvature we obtain
\begin{align*}
\ga\big(\mathcal R(&X,Y)X,Y\big) = \ga(\rh_X\rh_YX,Y) - \ga(\rh_Y\rh_XX,Y) + \ga(\rh_{[X,Y]}X,Y)
\\&\quad
-\frac12\ga(\rh_X[Y,X],Y) + \frac12\ga([Y,\rh_XX],Y) 
\\&\quad
+0 - \frac12\ga([X,\rh_YX],Y) -\frac14\ga([[X,Y],X],Y)
\\&
= \ga(\rh_XX,\rh_YY) - \|\rh_XY\|_\ga^2 + \frac34\|[X,Y]\|_\ga^2  
\\&\quad
-\frac12\ga(X,[Y,[X,Y]]) + \frac12\ga(Y,[X,[X,Y]])
\\&
= \ga(\rh_XX,\rh_YY) - \|\rh_XY\|_\ga^2 + \frac34\|[X,Y]\|_\ga^2  
\\&\quad
-\ga(\rh_XY,[X,Y]]) + \ga(Y,[X,[X,Y]]).
\end{align*}
If the adjoint $\on{ad}(X)^\top:\mathfrak g\to \mathfrak g$ exists, this is easily seen to coincide with 
Arnold's original formula \cite{Arnold66},
\begin{align*}
\ga(\mathcal R(X,Y)X,Y) =& 
- \frac14\|\on{ad}(X)^\top Y+\on{ad}(Y)^\top X\|^2_\ga
+ \ga(\on{ad}(X)^\top X,\on{ad}(Y)^\top Y)   
\\&
+ \frac12\ga(\on{ad}(X)^\top Y-\on{ad}(Y)^\top X,\on{ad}(X)Y)  
+ \frac34\|[X,Y]\|_\ga^2.
\end{align*}

\subsection{Weak right invariant Riemannian metrics on diffeomorphism groups}\label{wRmDiff}
Let $N$ be a finite dimensional manifold. We consider the following regular Lie groups:
$\on{Diff}(N)$, the  group of all diffeomorphisms of $N$ if $N$ is compact. 
$\on{Diff}_c(N)$, the group of diffeomorphisms with compact support, if $N$ is not compact.
If $N=\mathbb R^n$, we also may consider one of the following:
$\on{Diff}_{\mathcal S}(\mathbb R^n)$, the group of all diffeomorphisms which fall rapidly to the identity.
$\on{Diff}_{W^{\infty,p}}(\mathbb R^n)$, the group of all diffeomorphisms which are modelled on the space 
$W^{\infty,p}(\mathbb R^n)^n$, the intersection of all $W^{k,p}$-Sobolev spaces of vector fields.
The last type of groups works also for a {\it Riemannian manifold of bounded geometry} 
$(N,\bar g)$; see \cite{Eichhorn2007} for Sobolev spaces on them.  
In the following we write $\Diff_{\mathcal A}(N)$ for any of these groups.
The Lie algebras are the spaces $\X_{\mathcal A}(N)$ of vector fields,
where $\mathcal A\in \{C^\infty_c, \mathcal S, W^{\infty,p}\}$, 
with the negative of the usual bracket as Lie bracket.

A right invariant weak inner product on $\Diff_{\mathcal A}(N)$ is given by a smooth positive definite 
inner product $\ga$ on the Lie algebra $\X_{\mathcal A}(N)$ which is described by the operator 
$L=\check \ga:\X_{\mathcal A}(N)\to \X_{\mathcal A}(N)'$ and we shall denote its inverse by 
$K=L\i:L(\X_{\mathcal A}(N))\to \X_{\mathcal A}(N)$. Under suitable conditions on $L$ (like 
an elliptic coercive (pseudo) differential operator of high enough order) the operator $K$ turns 
out to be the reproducing kernel of a Hilbert space of vector fields which is contained in the 
space of either $C^1_b$ (bounded $C^1$ with respect to $\bar g$) or $C^2_b$ vector fields. 
See \cite[Chapter 12]{Younes2010}, \cite{Micheli2013}, and \cite{MumfordMichor13} for uses of 
the reproducing Hilbert space approach.  
The right invariant metric is then 
defined as in \ref{wRmLie}, where $\langle \;,\; \rangle_{\X_{\mathcal A}(N)}$  is the duality:
$$
G^L_\ph(X\o\ph, Y\o\ph) = G^L_{\on{Id}}(X,Y) = \ga(X,Y)  =  \langle L(X),Y 
\rangle_{\X_{\mathcal A}(N)}.
$$
Mis-using the notation for $L$ we will often also write  
$$
G^L_{\on{Id}}(X,Y) = \int_N \bar g(LX,Y)\vol(\bar g).
$$
Examples of metrics are:
\begin{align*}
G^0_{\on{Id}}(X,Y) &= \int_N \bar g(X,Y) \vol(\bar g) &\qquad&\text{the  }L^2\text{  metric,}
\\
G^s_{\on{Id}}(X,Y) &= \int_N \bar g((1+\De^{\bar g})^s X,Y) \vol(\bar g) &\qquad&
\text{a Sobolev metric of order }s,
\\
G^{\dot H^1}_{\on{Id}}(X,Y) &= \int_{\mathbb R} X'.Y' dx = -\int_{\mathbb R} X''Y\,dx &&
\text{where  }X,Y\in \X_{\mathcal A}(\mathbb R).
\end{align*}
The geodesic equation on $\Diff_{\mathcal A}(N)$. As explained in \ref{wRmLie}, the geodesic 
equation is given as follws:
Let $\ph:[a,b]\to \Diff_{\mathcal A}(N)$ be a smooth curve.
In terms of its right logarithmic derivative $u:[a,b]\to \X_{\mathcal A}(N)$, 
$u(t):= \ph^*\ka(\partial_t) = \ph'(t)\o \ph(t)\i$, 
the geodesic equation is
$$
L(u_t) = L(\p_t u)= - \on{ad}(u)^* L(u).
$$
The {\em condition for the existence of the geodesic equation} is as follows:
$$
X\mapsto K(\on{ad}(X)^*L(X))
$$
is bounded quadratic $\X_{\mathcal A}(N)\to \X_{\mathcal A}(N)$.
Using  {\it  Lie derivatives}, the computation of $\on{ad}_X^*$ is
especially  simple.  Namely,  for  any  section  
$\om$ of $T^*N \otimes  \on{vol}$   and vector fields 
$\xi,\et \in \X_{\mathcal A}(N)$, we have:
$$ \int_N (\om, [\xi,\et]) = \int_N (\om, \L_\xi(\et)) = 
-\int_N(\L_\xi(\om),\et),$$ 
hence $\on{ad}^*_\xi(\om) = +\L_\xi(\om)$. 
Thus the Hamiltonian version of the geodesic equation on the smooth dual 
$L(\X_{\mathcal A}(N))\subset \Ga_{C^2_b}(T^*N\otimes \on{vol})$ becomes
$$
\p_t\al  = - \on{ad}^*_{K(\al)}\al = - \L_{K(\al)}\al,
$$
or, keeping track of everything,
\begin{align*}
\p_t\ph &= u\o \ph, &\quad& 
\p_t\al = - \L_u\al &&
u = K(\al) = \al^\sharp,&&\al=L(u) = u^\flat.
\end{align*}

\begin{thm}
Geodesic distance vanishes on $\Diff_{\mathcal A}(N)$ for any Sobolev metric of order 
$s<\frac12$. If $N=S^1\x C$ with $C$ compact, then geodesic distance vanishes also for $s=\frac12$.
It also vanises for the $L^2$-metric on the Virasoro group 
$\mathbb R\rtimes \Diff_{\mathcal A}(\mathbb R)$.

Geodesic distance is positive on $\Diff_{\mathcal A}(N)$ for any Sobolev metric of order 
$s\ge1$. If $\dim(N)=1$ then geodesic distance is also positive for $s>\frac12$.
\end{thm}

This proved in \cite{Bauer2013b}, \cite{Bauer2013c}, and \cite{Bauer2012c}.
Note that low order Sobolev metrics have geodesic equations corresponding to well-known non-linear 
PDEs:
On $\Diff(S^1)$ or $\Diff_{\mathcal A}(\mathbb R)$ the $L^2$-geodesic equation is Burgers' equation, 
on the Virasoro group it is the KdV equation, and the (standard) $H^1$-geodesic is (in both cases a 
variant of) the Camassa-Holm equation; see 
\cite[7.2]{Bauer2014} for a more comprehensive overview. All these are completely integrable 
infinite dimensional Hamiltonian systems.

\begin{thm}
Let $(N,\bar g)$ be a compact Riemannian manifold.
Then the geodesic equation is locally well-posed on $\Diff_{\mathcal A}(N)$ and the geodesic exponential 
mapping is a local diffeomorphism for a Sobolev metric of 
integer order $s\ge 1$.
For a Sobolev metric of integer order $s>\frac{\dim(N)+3}{2}$ the geodesic equation is even globally 
well-posed, so that $(\Diff_{\mathcal A}(N), G^s)$ is geodesically complete. This is also true for 
non-integer order $s$ if $N=\mathbb R^n$.

For $N=S^1$, the geodesic equation is locally wellposed even for $s\ge \frac12$.
\end{thm}

For these results see \cite{Bauer2011b}, \cite{Escher2014}, \cite{Escher2014b}, \cite{Bauer2015}. 

\begin{thm}\label{Rmap}
{\rm \cite{BBM14b}}
For $\mathcal A\in\{C^\infty_c, \mathcal S, W^{\infty,1}\}$ 
let 
$$
\mathcal A_1(\mathbb R)=\{f\in C^\infty(R)\,:\, f'\in\mathcal A(\mathbb R)\,,\, 
f(-\infty)=0\}
$$
and let $\Diff_{\mathcal A_1}(\mathbb R)=\{\ph = \on{Id}+f\,:\, 
f\in \mathcal A_1(\mathbb R)\,,\, f'>-1\}$. These are all regular Lie groups. 
The right invariant weak Riemannian metric
$G^{\dot H^1}_{\on{Id}}(X,Y)= \int_{\mathbb R}X'Y'\,dx$
is positive definite both on $\Diff_{\mathcal A}(\mathbb R)$ where it does not admit a geodesic 
equation (a non-robust weak Riemannian manifold), and on 
$\Diff_{\mathcal A_1}(\mathbb R)$ where it admits a geodesic equation but not in the stronger sense 
of Arnold. 
On $\on{Diff}_{\mathcal A_1}(\mathbb R)$ the geodesic equation is the Hunter-Saxton equation
$$
(\ph_t)\circ\ph\i=u,  \qquad u_{t} = -u u_x +\frac12 \int_{-\infty}^x   (u_x(z))^2 \,dz\;,
$$
and the induced geodesic distance is positive. 
We define the $R$-map by:
$$ R:
 \Diff_{\mathcal A_1}(\mathbb R)\to 
\mathcal A\big(\mathbb R,\mathbb R_{>-2}\big)\subset\mathcal A(\mathbb R,\mathbb R),
\quad R(\ph) = 2\;\big((\ph')^{1/2}-1\big)\, .
$$
The $R$-map is invertible with inverse
$$R\i :
\mathcal A\big(\mathbb R,\mathbb R_{>-2}\big) \to \Diff_{\mathcal A_1}(\mathbb R),
\quad R\i(\ga)(x) = x+\frac14 \int_{-\infty}^x \ga^2+4\ga\;dx\; .
$$ 
The pull-back of the flat $L^2$-metric via $R$ is the $\dot H^1$-metric on $\on{Diff}_\mathcal A(\mathbb R)$, i.e., 
$R^*\langle \cdot,\cdot\rangle_{L^2(dx)} = G^{\dot H^1}$. 
Thus the space $\big(\Diff_{\mathcal A_1}(\mathbb R),\dot H^1\big)$ is a flat space in the sense of Riemannian geometry.
There are explicit formulas for geodesics, geodesic distance, and geodesic splines, even for more 
restrictive spaces $\mathcal A_1$ like Denjoy-Carleman ultradifferentiable function classes.
There are also soliton-like solutions. 
$(\Diff_{\mathcal A_1}(\mathbb R), G^{\dot H^1})$ is geodesically convex, but not geodesically 
complete; the geodesic completion is the smooth semigroup 
$\on{Mon}_{\mathcal A_1} = \{\ph=\on{Id}+f\,:\, f\in\mathcal A_1(\mathbb R)\,,\, f'\ge -1\}$.
Any geodesic can hit the subgroup 
$\Diff_{\mathcal A}(\mathbb R)\subset \Diff_{\mathcal A_1}(\mathbb R)$ at most twice.
\end{thm}

\section{Robust weak Riemannian manifolds and Riemannian submersions}

Another problem arises if we want to consider Riemannian submersions, in particular shape spaces as 
orbits of diffeomorphism groups, as explained in \cite{Michor127}.

\subsection{Robust weak Riemannian manifolds}
Some  constructions may lead to vector fields whose values do not lie in
$T_x M$, but in the Hilbert space completion $\overline{T_x M}$
with respect to the weak inner product $g_x$. 
We  need that
$\bigcup_{x\in M}\overline{T_x M}$ forms a smooth vector bundle
over  $M$.  In a coordinate chart on open
$U \subset M$, $TM|_U$ is a trivial bundle $U \times V$
and all the inner products $g_x, x \in U$ define inner products
on the  same topological vector space $V$. They all should be 
bounded with respect to each other, so that
the completion $\overline{V}$ of $V$ with respect to $g_x$ does
not depend on $x$ and $\bigcup_{x\in U}\overline{T_x M} \cong U
\times    \overline{V}$.   This   means   that   $\bigcup_{x\in
M}\overline{T_x  M}$ forms a smooth vector bundle over $M$ with
trivialisations the linear extensions of the trivialisations of
the  bundle  $TM\to M$. Chart changes should respect this. This is a compatibility 
property between the weak Riemannian metric and some smooth atlas of $M$.

{\bf Definition}  A convenient weak Riemannian manifold
$(M,g)$ will be called a {\em robust} Riemannian manifold if
\begin{itemize}
\item The Levi-Civita convariant derivative of the metric $g$ exists: 
The symmetric gradients should exist and be smooth.
\item The completions $\overline{T_x M}$ form a smooth vector bundle as above.
\end{itemize}

\begin{thm} If a right invariant weak Riemannian metric on a regular Lie group admits the Levi-Civita 
covariant derivative, then it is already robust. 
\end{thm}

\begin{proof}
By right invariance, each right translation $T\mu^g$ extends to an isometric isomorphims 
$\overline{T_xG} \to \overline{T_{xg}G}$. By the uniform boundedness theorem these isomorphisms 
depend smoothly on $g\in G$.
\end{proof}

\subsection{Covariant curvature and O'Neill's formula}
In \cite[2.2]{Michor127} one finds the following 
formula for the numerator of sectional curvature,
which is valid for {\em closed smooth} 1-forms $\al,\be \in\Om^1_g(M)$ on a weak Riemannianian 
manifold $(M,g)$. Recall that we 
view $g:TM\to T^*M$ and so $g\i$ 
is the dual inner product on $g(TM)$ and $\al^\sharp = g\i(\al)$. 
$$
\begin{aligned}
&g\big(R(\al^{\sharp},\be^{\sharp})\al^{\sharp},\be^{\sharp}\big) =
\\&
-\tfrac12\al^{\sharp}\al^{\sharp}(\|\be\|_{g\i}^2)
-\tfrac12\be^{\sharp}\be^{\sharp}(\|\al\|_{g\i}^2)
+\tfrac12(\al^{\sharp}\be^{\sharp}+\be^{\sharp}\al^{\sharp})g\i(\al,\be)
\\&\qquad\qquad
\big(\text{last line }= -\al^\sharp \be([\al^\sharp,\be^\sharp])+\be^\sharp\al([\al^\sharp,\be^\sharp]])\big)
\\&
-\tfrac14\|d(g\i(\al,\be))\|_{g\i}^2
+\tfrac14 g\i\big(d(\|\al\|_{g\i}^2),d(\|\be\|_{g\i}^2)\big)
\\&
+\tfrac34 \big\|[\al^{\sharp},\be^{\sharp}]\big\|_{g}^2
\end{aligned}
$$
This is called Mario's formula since Mario Micheli derived the coordinate version in his 2008 
thesis. Each term depends only on $g\i$ with the exception of the last term. The role of the last 
term (which we call the O'Neill term) will become clear  in the next result.
Let $p:(E,g_E)\to (B,g_B)$ be a Riemannian submersion between infinite dimensional 
robust Riemann manifolds; i.e., for each $b\in B$ and $x\in E_b:=p\i(b)$ 
the tangent mapping $T_xp:(T_xE, g_E)\to (T_bB,g_B)$ 
is a surjective metric quotient map so that 
\begin{equation*}
\|\xi_b\|_{g_B} := \inf\bigl\{\|X_x\|_{g_E}\,:\,X_x\in T_xE, T_xp.X_x=\xi_b\bigr\}.
\end{equation*}
The infinimum need not be attained in $T_xE$ but will be in the completion 
$\overline{T_xE}$. 
The orthogonal subspace $\{Y_x: g_E(Y_x,T_x(E_b))=0\}$ 
will therefore be taken in $\overline{T_x(E_b)}$ in $T_xE$. 
If $\al_b=g_B(\al_b^\sharp,\quad)\in g_B(T_bB)\subset T_b^*B$ 
is an element in the $g_B$-smooth dual, then 
$p^*\al_b:=(T_xp)^*(\al_b)= g_B(\al_b^\sharp,T_xp\quad):T_xE\to \mathbb R$ 
is in $T_x^*E$ but in general it is not an element in the smooth dual $g_E(T_xE)$. 
It is, however, an element of the Hilbert space completion $\overline{g_E(T_xE)}$ 
of the $g_E$-smooth dual $g_E(T_xE)$ with respect to the norm $\|\quad\|_{g_E\i}$, 
and the element $g_E\i(p^*\al_b)=: (p^*\al_b)^\sharp$ is in the 
$\|\quad\|_{g_E}$-completion $\overline{T_xE}$ of $T_xE$. 
We can call $g_E\i(p^*\al_b)=: (p^*\al_b)^\sharp$ the {\it horizontal lift} of 
$\al_b^\sharp = g_B\i(\al_b)\in T_bB$.

\begin{thm}
{\rm \cite[2.6]{Michor127}}
Let $p:(E,g_E)\to (B,g_B)$ be a Riemann submersion between infinite
dimensional robust Riemann manifolds. Then for {\bf closed} 1-forms
$\al,\be\in\Om_{g_B}^1(B)$ O'Neill's formula holds in the form:
\begin{align*}
 g_B\big(R^B(\al^{\sharp},\be^{\sharp})\be^{\sharp},\al^{\sharp}\big)
&= g_E\big(R^E((p^*\al)^{\sharp},(p^*\be)^{\sharp})(p^*\be)^{\sharp},(p^*\al)^{\sharp}\big)
\\&\quad
+\tfrac34\|[(p^*\al)^{\sharp},(p^*\be)^{\sharp}]^{\text{ver}}\|_{g_E}^2
\end{align*}
\end{thm}

\begin{proof} The last (O'Neill) term is the difference between curvature on $E$ and the pullback 
of the curvature on $B$.
\end{proof}

\subsection{Semilocal version of Mario's formula, force, and stress}
In all interesting examples of orbits of diffeomorphisms groups through a template shape, 
Mario's covariant curvature formula leads to complicated and impenetrable formulas. 
Efforts to break this down to comprehensible pieces led to the concepts of symmetrized force and 
(shape-) stress explained below. Since acceleration sits in the second tangent bundle, one either 
needs a covariant derivative to map it down to the tangent bundle, or at least rudiments of  local 
charts. In \cite{Michor127} we managed the local version. Interpretations in mechanics or 
elasticity theory are still lacking. 

Let $(M,g)$ be a robust Riemannian manifold, $x\in M$, $\al,\be\in g_x(T_xM)$.
Assume we are given local smooth vector fields $X_\al$ and $X_\be$ such that:
\begin{enumerate}
\item $X_\al(x) = \al^\sharp(x),\quad X_\be(x) = \be^\sharp(x)$,
\item Then $\al^\sharp-X_\al$ is zero at $x$ hence has a well defined derivative 
       $D_x(\al^\sharp-X_\al)$ lying in Hom$(T_xM,T_xM)$. For a vector field $Y$ we have 
       $D_x(\al^\sharp-X_\al).Y_x = [Y,\al^\sharp-X_\al](x) = \L_Y(\al^\sharp-X_\al)|_x$.
       The same holds for $\be$.
\item $\L_{X_\al}(\al)=\L_{X_\al}(\be)=\L_{X_\be}(\al)=\L_{X_\be}(\be)=0$,
\item $[X_\al, X_\be] = 0$.
\end{enumerate}
Locally constant 1-forms and vector fields will do. We then define:
\begin{align*}
\mathcal F(\al,\be) :&= \tfrac12 d(g\i(\al,\be)), \qquad \text{a 1-form on $M$ called the {\it force},}\\
\mathcal D(\al,\be)(x) :&= D_x(\be^\sharp - X_\be).\al^\sharp(x)
\\&
= d(\be^\sharp - X_\be).\al^\sharp(x), \quad\in T_xM\text{ called the {\it stress}.}
\\
\implies &\mathcal D(\al,\be)(x) - \mathcal D(\be,\al)(x) = [\al^\sharp,\be^\sharp](x)
\end{align*}  
Then in terms of force and stress the numerator of sectional curvature looks as follows:
\begin{align*}
g\big(&R(\al^{\sharp},\be^{\sharp})\be^{\sharp},\al^{\sharp}\big)(x) = R_{11} +R_{12} + R_2 + 
R_3\,, 
\qquad\text{  where }
\\ 
& R_{11} = \tfrac12 \big(
\L_{X_\al}^2(g\i)(\be,\be)-2\L_{X_\al}\L_{X_\be}(g\i)(\al,\be)
+\L_{X_\be}^2(g\i)(\al,\al)
 \big)(x)\,, 
\\
& R_{12} = \langle \mathcal F(\al,\al), \mathcal D(\be,\be) \rangle + \langle \mathcal F(\be,\be),\mathcal D(\al,\al)\rangle  
- \langle \mathcal F(\al,\be), \mathcal D(\al,\be)+\mathcal D(\be,\al) \rangle \,,
\\
& R_2 = \big(
\|\mathcal F(\al,\be)\|^2_{g\i}
-\big\langle \mathcal F(\al,\al)),\mathcal F(\be,\be)\big\rangle_{g\i} \big)(x)\,, 
\\
& R_3 = -\tfrac34 \| \mathcal D(\al,\be)-\mathcal D(\be,\al) \|^2_{g_x} \,.
\end{align*}

\subsection{Landmark space as homogeneus space of solitons}
This subsection is based on \cite{Micheli2012};
the method explained here has many applications in computational anatomy and elswhere, under the 
name LDDMM (large diffeomorphic deformation metric matching).

A {\em landmark} $q=(q_1,\dots,q_N)$ is an $N$-tuple 
of distinct points in $\mathbb R^n$;    
landmark space $\on{Land}^N(\mathbb R^n)\subset (\mathbb R^n)^N$ is open.
Let $q^0=(q^0_1,\dots,q^0_N)$ be a fixed standard template 
landmark. Then we have the surjective mapping
\begin{align*}
&\on{ev}_{q^0}:\on{Diff}_{\mathcal A}(\mathbb R^n)\to \on{Land}^N(\mathbb R^n),
\\&
\ph\mapsto \on{ev}_{q^0}(\ph)=\ph(q^0)=(\ph(q^0_1),\dots,\ph(q^0_N)).
\end{align*}
Given a Sobolev metric of order $s>\frac{n}2 + 2$ on $\Diff_{\mathcal A}(\mathbb R^n)$,
we want to induce a Riemannian metric on $\on{Land}^N(\mathbb R^n)$ such that $\on{ev}_{q_0}$ becomes a 
Riemannian submersion. 

The fiber of $\on{ev}_{q^0}$ over a landmark $q=\ph_0(q^0)$ is  
\begin{align*}
\{\ph\in\on{Diff}_{\mathcal A}(\mathbb R^n): \ph(q^0)=q\}
&= \ph_0\o\{\ph\in\on{Diff}_{\mathcal A}(\mathbb R^n): \ph(q^0)=q^0\}
\\&
= \{\ph\in\on{Diff}_{\mathcal A}(\mathbb R^n): \ph(q)=q\}\o\ph_0\,.
\end{align*}
The tangent space to the fiber is 
$$
\{X\o \ph_0: X\in \X_{\mathcal S}(\mathbb R^n), X(q_i) = 0 \text{  for all } i \}.
$$
A tangent vector $Y\o\ph_0 \in T_{\ph_0}\on{Diff}_{\mathcal S}(\mathbb R^n)$ is 
$G_{\ph_0}^L$-perpendicular to the fiber over $q$ if and only if
$$
\int_{\mathbb R^n} \langle LY, X \rangle\,dx =0\quad 
\forall X\text{  with }X(q)=0.
$$
If we require $Y$ to be smooth then $Y=0$. So we assume that 
$LY=\sum_i P_i.\de_{q_i}$, a distributional vector field with support
in $q$. Here $P_i\in T_{q_i}\mathbb R^n$. 
But then  
\begin{align*}
Y(x) &= L\i\Big(\sum_i P_i.\de_{q_i}\Big) 
= \int_{\mathbb R^n} K(x-y)\sum_i P_i.\de_{q_i}(y)\,dy
= \sum_i K(x-q_i).P_i\,,
\\
&T_{\ph_0}(\on{ev}_{q^0}).(Y\o\ph_0) = Y(q_k)_k = \sum_i (K(q_k-q_i).P_i)_k\,.
\end{align*}
Now let us consider a tangent vector $P=(P_k)\in T_q\on{Land}^N(\mathbb R^n)$. Its
horizontal lift with footpoint $\ph_0$ is $P^{\text{hor}}\o\ph_0$
where the vector field $P^{\text{hor}}$ on $\mathbb R^n$ is given as follows:
Let $K\i(q)_{ki}$ be the inverse of the $(N\x N)$-matrix $K(q)_{ij}=K(q_i-q_j)$.
Then  
\begin{align*}
P^{\text{hor}}(x) &= \sum_{i,j} K(x-q_i)K\i(q)_{ij}P_j\,,
\\
L(P^{\text{hor}}(x))&= \sum_{i,j} \de(x-q_i)K\i(q)_{ij}P_j\,.
\end{align*}
Note that $P^{\text{hor}}$ is a vector field of class $H^{2l-1}$.

The Riemannian metric on $\on{Land}^N$ induced by the $g^L$-metric on
$\on{Diff}_{\mathcal S}(\mathbb R^n)$ is 
\begin{align*}
g^L_q(P,Q) &= G^L_{\ph_0}(P^{\text{hor}}, Q^{\text{hor}})
= \int_{\mathbb R^n}\langle L(P^{\text{hor}}),Q^{\text{hor}} \rangle\,dx
\\&
= \int_{\mathbb R^n}\Big\langle \sum_{i,j}\de(x-q_i)K\i(q)_{ij}P_j,
 \sum_{k,l} K(x-q_k)K\i(q)_{kl}Q_l \Big\rangle\,dx
\\&
= \sum_{i,j,k,l} K\i(q)_{ij}K(q_i-q_k)K\i(q)_{kl}\langle P_j,Q_l \rangle
\\
g^L_q(P,Q) &= \sum_{k,l} K\i(q)_{kl}\langle P_k,Q_l \rangle.
\end{align*}
The {\em geodesic equation} in vector form is:  
\begin{align*}
\ddot q_n =
& -\frac12
\sum_{k,i,j,l}K\i(q)_{ki}
\on{grad} K(q_i-q_j)(K(q)_{in}-K(q)_{jn}) 
K\i(q)_{jl}\langle \dot q_k,\dot q_l \rangle
\\&
+ \sum_{k,i}K\i(q)_{ki}
\Big\langle\on{grad} K(q_i-q_n),\dot q_i-\dot q_n\Big\rangle 
\dot q_k\,.
\end{align*}
The cotangent bundle  $T^*{\on{Land}^N(\mathbb R^n)}=  
\on{Land}^N(\mathbb R^n)\x ((\mathbb R^n)^N)^*\ni (q,\al)$. 
We shall treat $\mathbb R^n$ like scalars;
$\langle\quad,\quad  \rangle$ is always the standard inner product on $\mathbb R^n$.
\\
The metric looks like 
\begin{align*}
(g^L)\i_q(\al,\be) &= \sum_{i,j} K(q)_{ij}\langle \al_i,\be_j \rangle,
\qquad
K(q)_{ij} = K(q_i-q_j).
\end{align*}
The energy function is
\begin{align*}
E(q,\al)&=\tfrac12 (g^L)\i_q(\al,\al) 
= \tfrac12\sum_{i,j} K(q)_{ij}\langle \al_i,\al_j \rangle
\end{align*}
and its Hamiltonian vector field (using $\mathbb R^n$-valued derivatives to save notation)  is
\begin{align*}
H_E(q,\al) &=
  \sum_{i,k=1}^N \Big(K(q_k-q_i)\al_i\frac{\p}{\p q_k} 
+ \on{grad} K(q_i-q_k)\langle \al_i,\al_k \rangle\frac{\p}{\p \al_k}\Big). 
\end{align*}
So the {\em Hamitonian version of the geodesic equation} is the flow of this vector field:
$$
\begin{cases}
\dot q_k &= \sum_i K(q_i-q_k)\al_i
\\
\dot \al_k &= -\sum_i  \on{grad} K(q_i-q_k) \langle \al_i,\al_k \rangle
\end{cases}
$$
We shall use {\em stress and force} to express the geodesic equation and curvature:
\begin{align*}
\al^\sharp_k =\sum_i &K(q_k-q_i)\al_i,
\quad \al^\sharp = \sum_{i,k} K(q_k-q_i)\langle \al_i,\tfrac{\p}{\p q^k} \rangle
\\
\mathcal D(\al,\be) :&
= \sum_{i,j} 
dK(q_i-q_j)(\al^\sharp_i-\al^\sharp_j)\Big\langle \be_j,\frac{\partial}{\partial q_i}\Big\rangle, 
\quad\text{the {\it stress.}}
\\
\mathcal D(\al,\be) &- \mathcal D(\be,\al) = (D_{\al^\sharp}\be^\sharp) - 
D_{\be^\sharp}\al^\sharp =[\al^\sharp,\be^\sharp],
\quad\text{{\it Lie bracket.}}
\\
\mathcal F_i(\al,\be) &= \frac12\sum_k \on{grad} K(q_i-q_k)(\langle \al_i,\be_k \rangle + \langle \be_i,\al_k \rangle )
\\
\mathcal F(\al,\be) :&= \sum_i \langle \mathcal F_i(\al,\be),dq_i\rangle = \frac12\, d\,g\i(\al,\be) 
\quad\text{the {\it force}.} 
\end{align*}
The geodesic equation on $T^*\on{Land}^N(\mathbb R^n)$ then becomes 
$$
\begin{cases}
\dot q &= \al^\sharp 
\\
\dot \al &= - \mathcal F(\al,\al) \,.
\end{cases}
$$
Next we shall compute {\em curvature via the cotangent bundle}.
From the semilocal version of Mario's formula for the numerator of the sectional curvature
for constant 1-forms $\al,\be$ on landmark space, where 
$\al^\sharp_k =\sum_i K(q_k-q_i)\al_i$, we get directly:
\begin{align*}
&g^L\big(R(\al^{\sharp},\be^{\sharp})\al^{\sharp},\be^{\sharp}\big) =
\\&=
\big\langle \mathcal D(\al,\be) + \mathcal D(\be,\al),\mathcal F(\al,\be)\big\rangle 
\\&\quad
-\big\langle \mathcal D(\al,\al),\mathcal F(\be,\be)\big\rangle 
-\big\langle \mathcal D(\be,\be),\mathcal F(\al,\al)\big\rangle
\\&\quad
-\tfrac12\sum_{i,j}\Big(
d^2K(q_i-q_j)(\be^\sharp_i-\be^\sharp_j,\be^\sharp_i-\be^\sharp_j)\langle \al_i,\al_j\rangle
\\&\qquad\qquad
-2d^2K(q_i-q_j)(\be^\sharp_i-\be^\sharp_j,\al^\sharp_i-\al^\sharp_j)\langle \be_i,\al_j\rangle
\\&\qquad\qquad
+d^2K(q_i-q_j)(\al^\sharp_i-\al^\sharp_j,\al^\sharp_i-\al^\sharp_j)\langle \be_i,\be_j\rangle
\Big)
\\&\quad
-\|\mathcal F(\al,\be)\|_{g\i}^2 + g\i\bigl(\mathcal F(\al,\al),\mathcal F(\be,\be)\bigr). 
\\&\quad
+\tfrac34 \| [\al^\sharp,\be^\sharp] \|_g^2
\end{align*}

\subsection{Shape spaces of submanifolds as homogeneous spaces for the diffeomorphism group}
Let $M$ be a compact manifold and $(N,\bar g)$ a Riemannian manifold of bounded geometry as in 
subsection \ref{wRmDiff}.
The diffeomorphism group $\Diff_{\mathcal A}(N)$ acts also from the left on the manifold of 
$\Emb(M,N)$ embeddings and also on the {\em non-linear Grassmannian} or 
{\em differentiable Chow variety} $B(M,N)=\Emb(M,N)/\Diff(M)$. 
For a Sobolev metric of order $s>\frac{\dim(N)}2+ 2$ one can then again induce a Riemannian metric 
on each $\Diff_{\mathcal A}(N)$-orbit, as we did above for landmark spaces. 
This is done in \cite{Michor127}, where the geodesic equation is computed and where curvature is 
described in terms of stress and force. 

\small\parskip=0pt
\def\cprime{$'$} \def\cprime{$'$}

\end{document}